\documentclass{amsart}  
\usepackage{graphicx}    
\usepackage{amsmath}
\usepackage{amssymb}
\usepackage{amsthm}
\usepackage{a4wide}
\usepackage{xcolor}
\theoremstyle{plain}
\newtheorem{thm}{Theorem}
\newtheorem{lem}[thm]{Lemma}

\newtheorem{cor}[thm]{Corollary}

\theoremstyle{definition}

\newtheorem{rem}{Remark}
\newtheorem{quest}{Question}

\newcommand{\R}{\mathbb{R}}

\newcommand{\N}{\ensuremath{\mathbb{N}}}
\newcommand{\T}{\ensuremath{{\mathcal T}}}

\renewcommand{\epsilon}{\varepsilon}

\newcommand{\grey}[1]{\textcolor[rgb]{0.5,0.5,0.5}{#1}}

\begin{document}


\title{Incongruent equipartitions of the plane}
\author{Dirk Frettl{\"o}h}
\address{Faculty of Technology, Bielefeld University, 33501 Bielefeld, Germany}
\author{Christian Richter}
\address{Institute of Mathematics, Friedrich Schiller University, 07737 Jena, Germany}

\begin{abstract}
R.\ Nandakumar asked whether there is a tiling of the plane by pairwise 
incongruent triangles of equal area and equal perimeter. Recently 
a negative answer was given by Kupavskii, Pach and Tardos. Still one
may ask for weaker versions of the problem, or for the analogue of
this problem for quadrangles, pentagons, or hexagons. Several answers
were given by the first author in a previous paper. Here we solve 
three further cases. In particular, our main result shows that
there are vertex-to-vertex tilings by pairwise incongruent triangles 
of unit area and bounded perimeter.

\end{abstract}


\maketitle


\section{Introduction}

Discrete geometry is one of the fields containing harmless and natural 
sounding questions where the answer may require a lot of effort.
R.\ Nandakumar posed several such intriguing questions about discrete
geometry in his blog \cite{nblog1}. Some of them have triggered a lot 
of research recently.
For example he asked: ``Is it possible to tile the plane with pairwise noncongruent 
equilateral triangles whose side lengths are bounded from below by a positive 
constant?" The answer is ``No" as was shown recently in \cite{pt,rw}. The 
corresponding answer for squares is more classical, see for instance \cite{gs}.
A very fruitful question of Nandakumar is ``can any convex set in the plane
be dissected into $n$ convex pieces with the same area and the same
perimeter?"  The problem seems harmless, but its solution required pretty
sophisticated tools from algebraic topology \cite{bbs, bz, kha, nr}. For a survey
see \cite{z2}. This paper is dedicated to another of his problems \cite{nblog1}: 

\begin{quest} \label{q:frage0}
``Can the plane be tiled by triangles of same 
area and perimeter such that no two triangles are congruent to each other?'' 
\end{quest}
Here congruence is meant with respect to Euclidean isometries; that is, 
isometries of the plane including reflections.
Question \ref{q:frage0} was answered in \cite{kpt1} by showing that no such tiling 
exists. Weakening the problem by dropping any requirement on the perimeter
makes the problem easy: it is not hard to find tilings of the plane by triangles of unit 
area with unbounded perimeter, see \cite{nblog1}. Hence Nandakumar also asked 
the following weaker version of Question \ref{q:frage0}: 
\begin{quest} \label{q:frage}
``Can the plane be tiled by triangles of same area, and with
uniformly bounded perimeter, such that no two triangles are congruent to each other?'' 
\end{quest}
As above, congruence is meant with respect to Euclidean isometries.
Question \ref{q:frage} was answered in \cite{f} (partly), and in \cite{kpt2}. 

\begin{thm}[\cite{f,kpt2}] \label{thm:fkpt} 
There are tilings of the plane by pairwise incongruent triangles 
of unit area and bounded perimeter. 
\end{thm}

Once this is settled one may ask the same
question for convex $n$-gons. Moreover, the tilings in \cite{f,kpt2} 
are not vertex-to-vertex (vtv), hence they are not polytopal cell 
decompositions \cite{z}. Hence one may ask whether there are vtv 
tilings fulfilling the properties of Theorem \ref{thm:fkpt}. 
Table \ref{tab:alles}
provides an overview of several variants of the problems, together
with the known answers. 
\begin{table}[b]
\begin{tabular}{l|cc}
\hline
Triangles & vtv &  not vtv \\
\hline
bounded perimeter & Yes & \grey{Yes$^{1,2}$}\\
tiling a tile & ? & ? \\
equal perimeter & \grey{No$^1$} & \grey{No$^3$}\\
\hline
\hline
Pentagons & vtv &  not vtv \\
\hline
bounded perimeter & Yes &  \grey{Yes$^1$}\\
tiling a tile & Yes &  \grey{Yes$^1$}\\
equal perimeter & ? & ?\\
\hline
\end{tabular}
\quad
\begin{tabular}{l|cc}
\hline
Quadrangles & vtv &  not vtv \\
\hline
bounded perimeter &  \grey{Yes$^1$} &  \grey{Yes$^1$}\\
tiling a tile &  \grey{Yes$^1$} &  \grey{Yes$^1$}\\
equal perimeter & ? & ?\\
\hline
\hline
Hexagons & vtv & not vtv\\
\hline
bounded perimeter &  \grey{Yes$^1$} & Yes \\
tiling a tile &  \grey{Yes$^1$} & ? \\
equal perimeter & ? & ?\\
\hline
\end{tabular}
\medskip
\caption{Several instances of the problem ``tiling the plane with 
pairwise incongruent convex $n$-gons of unit area'' plus some further
conditions. Grey entries marked $^1$ were solved in \cite{f}. The grey 
entry marked $^{1,2}$ was solved in \cite{f} (partly) and in \cite{kpt2}. 
The grey entry marked $^3$ was solved in \cite{kpt1}. Black entries ``Yes'' are 
solved in this paper. \label{tab:alles}}
\end{table}

This paper is devoted to solve further instances of the problem
(shown in black in Table \ref{tab:alles}).
In the remainder of this section we introduce some notations and
basic ideas. Along the way we show that there are vertex-to-vertex 
tilings of the plane by pairwise incongruent convex pentagons of 
unit area and bounded perimeter (Theorem \ref{thm:pent}). 
The main result of this paper is Theorem \ref{thm:vtv-triangles}.
It says that there are vertex-to-vertex tilings of the plane by
pairwise incongruent triangles of unit area and bounded perimeter.
Section \ref{sec:tritil} is devoted to the proof of this result.
In the last section we construct a 
tiling of the plane by pairwise incongruent convex hexagons of 
unit area and bounded perimeter that are \emph{not} vertex-to-vertex.

Note that usually the requirement of being vertex-to-vertex is
more restrictive than the requirement of being not vertex-to-vertex,
but for convex hexagons it is the other way around.
For instance, it is easy to produce tilings by unit squares that
are not vertex-to-vertex: consider a regular tiling by unit
squares and shift one column of squares by $1/2$. In fact,
the only vertex-to-vertex tiling of $\R^2$ by unit squares
is the regular one (up to conguence); and there are uncountably 
many distinct tilings of $\R^2$ by unit squares with infinitely many
non-vtv situations.

For normal tilings (that is, tilings where the tiles do not get arbitrarily 
small, and the perimeter of the tiles does not get arbitrarily large \cite{gs})
of the plane by convex hexagons it is known that  
they cannot contain infinitely many non-vtv situations.
Try it, or see \cite{a,fgl}.

\subsection{Notation}
A {\em tiling} of a set $A \subseteq \R^2$ is a collection $\{T_1, T_2, \ldots \}$ 
of compact sets $T_i \subseteq \R^2$ (the {\em tiles}) that is a packing 
(that is, the interiors of distinct tiles are disjoint) as well as a 
covering of $A$ (that is, the union of the tiles equals $A$).
In general, shapes of tiles may be pretty complicated, but for
the purpose of this paper tiles are always convex polygons. 
A tiling is called {\em vertex-to-vertex (vtv)}, if the intersection of any two 
tiles is either an entire edge of both tiles, or a vertex of both tiles, or 
empty. 
An \emph{equipartition} of the plane is a tiling of $\R^2$ such that
each tile has unit area. The set of positive integers is denoted by $\N$.
We refer to \cite{gs} as a standard reference work on tilings.


\subsection{Tiling a tile} \label{sec:tiling-a-tile}

One possible approach to find a solution for an entry in Table \ref{tab:alles} 
is the following. If one can partition a set $P \subseteq \R^2$ into 
convex $n$-gons of unit area such that 
\begin{itemize}
\item[(i)] $P$ tiles the plane
(that is, $\mathbb{R}^2$ can be tiled by congruent copies of $P$), and
\item[(ii)] all $n$-gons in $P$ can 
be distorted continuously, in a way such that any two $n$-gons in 
$P$ are incongruent (but still have unit area), 
\end{itemize}then this yields a 
solution to the problem under consideration. Let us illustrate this 
concept for pentagons. Figure~\ref{fig:5gons} shows a large
pentagon $P$ that tiles the plane.
This large pentagon is divided into six smaller pentagons
of unit area (left). Four of the five inner vertices of this subdivision 
can be wiggled continuously without destroying the unit area
property:
\begin{figure}
\[ \includegraphics[width=.55\textwidth]{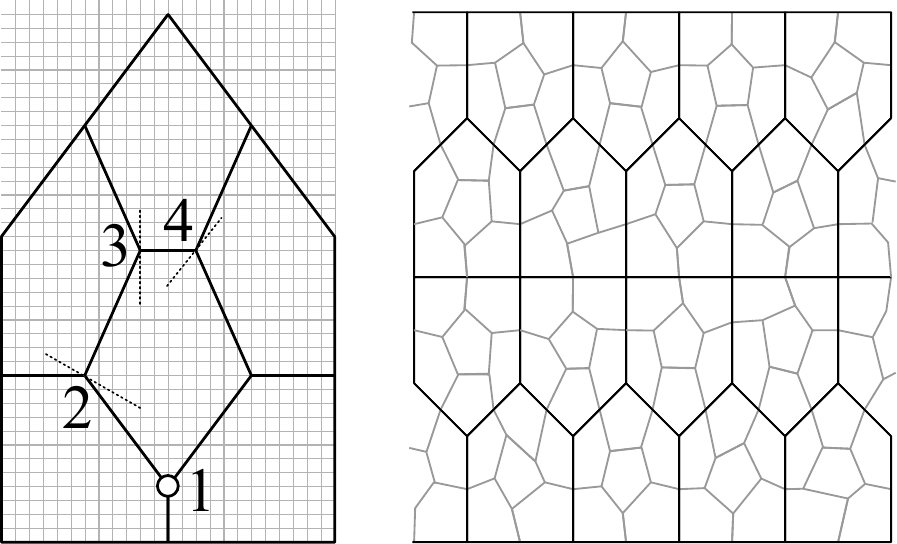} \]
\caption{Partitioning a convex pentagon into six distinct convex
  pentagons of equal area. A white dot (resp., a dashed line) indicates two
  degrees (resp.,  one degree) of freedom for moving the corresponding
  vertex. The numbers give 
the order of choosing free vertices. Unmarked vertices are not free. 
The resulting tiling is vertex-to-vertex. \label{fig:5gons}}
\end{figure}
As indicated in the image,
there are two continuous degrees of freedom for the choice
of the first vertex, marked by 1 in the image.
That is, vertex 1 can be moved within some ball of small radius. Once this
vertex is fixed, there is just one line segment representing the 
possible positions of vertex 2 such that the area 
of the lower left pentagon remains one. That is, vertex 2 
will still have one degree of freedom (indicated by a dashed line). 
In a similar manner now vertex 3 needs to be adjusted 
such that the area of the pentagon left of it remains one. This 
vertex still has one degree of freedom: it can
be shifted in direction of the dashed line without affecting the 
area of the pentagon left of it. The same is true for 
vertex 4: it will be affected by the choice of vertex 3 in order to keep
the area of the upper pentagon to one, but still has one
degree of freedom, indicated by a dashed line.
When all these four vertices are chosen, the last interior
vertex is fixed by the condition that both the two right-hand 
interior pentagons need to have area one: it is the intersection point of 
two segments representing admissible positions such that the areas of the 
upper right and the lower right interior pentagons are one, respectively. 
The central pentagon will automatically have unit area, since the areas of 
the six small pentagons add up to $6$.

Hence we may partition the large pentagon (of area $6$, say) in uncountably
many ways into pairwise incongruent pentagons
of unit area. This yields the desired equipartition of the plane:
In the first large pentagon we will distort the pentagons in a way
such that all of them are incongruent. This can be achieved
easily. In the second large pentagon we use the uncountably 
many possible choices in order to achieve that all small 
pentagons in the second large pentagon are 
\begin{itemize}
\item[(i)] incongruent to all pentagons in the first large pentagon and 
\item[(ii)]  pairwise incongruent to each other. 
\end{itemize}
And so on. The large pentagons will 
be arranged as on the right in Figure \ref{fig:5gons}.  
At each stage there are only finitely
many shapes of small pentagons to be avoided. Because we may
choose from uncountably many pentagons, this procedure
yields the desired equipartition of the plane into pentagons.
Clearly these equipartitions are vertex-to-vertex. This way tiling a 
tile gives the following result.
\begin{thm} \label{thm:pent}
There is a vertex-to-vertex tiling of the plane by pairwise incongruent 
convex pentagons of unit area and uniformly bounded perimeter.
\end{thm}
In the next section we will prove a similar result for triangles,
but the proof requires more effort. 


\section{Vtv equipartitions of the plane into triangles} \label{sec:tritil}

The argument that we can use continuous degrees of freedom
in order to avoid finitely many (or countably many) shapes will
also be used in the proof of the following result. 

\begin{thm} \label{thm:vtv-triangles}
There is a vertex-to-vertex tiling of the plane by pairwise incongruent 
triangles of unit area and uniformly bounded perimeter.
\end{thm}

\begin{proof}
The general idea of the construction is the following: Consider the
strip $S=\R \times [-1,1]$. Tile $S$ by (pairwise congruent) triangles of 
unit area with edge lengths $\sqrt{2},\sqrt{2}$ and $2$, see Figure~\ref{fig:streifen-ungestoert}.  
\begin{figure}[t]
\[ \includegraphics[width=.8\textwidth]{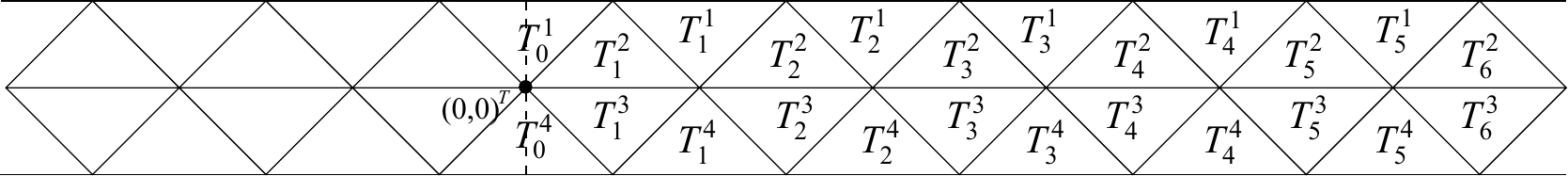} \]
\caption{A tiling of the strip $S$ by pairwise congruent triangles.
\label{fig:streifen-ungestoert}}
\end{figure}
Distort the tiling of $S$ by moving the vertex 
at $(0,0)^T$ to $(0,y_0)^T$ for $0<y_0<1$,
see Figure~\ref{fig:streifen-bez}.

Under the conditions that 
\begin{itemize}
\item[(i)] the topology of the tiling is unchanged, 
\item[(ii)] the new tiling still is a tiling of $S$, 
\item[(iii)] the new tiling is mirror symmetric with respect to the vertical
axis $x=0$, and 
\item[(iv)] the tiles of the new tiling stay triangles of unit area, 
\end{itemize}
the value of $y_0$ determines all other vertices of the tiling.
See Figure \ref{fig:streifen-bez} for the situation where $y_0=\frac{3}{5}$.

In the sequel the strategy of the proof is as follows: first we obtain recursive
formulas for the coordinates of the triangles in the tiling of the strip when
$y_0$ varies, in order to control the amount of distortion of the triangles (Lemma
\ref{lem:xi-yi-ai-bi}). This ensures in particular that the perimeter of the triangles
stays uniformly bounded (Lemma \ref{lem:perimeter}). Then we study the tiling for
$y_0 = \frac{1}{\sqrt{3}}$, having the particular property that it contains many pairwise
congruent triangles $T \cong T'$; and even stronger: it contains many triangles $T,T'$
such that $\pm T$ is a translate of $T'$ (see Figure \ref{fig:1/sqrt(3)},
this property is denoted by $T \simeq T'$).
This is Lemma \ref{lem:1/sqrt(3)}. This can be used in Lemmas \ref{lem:F(i,j,i',j')}
and \ref{lem:Ffinite} to show that there are only countably many $y_0$ such that
the corresponding tiling contains triangles $T,T'$ such that $T \simeq T'$. This in turn
enables us to pick a tiling $\overline{\T}$ of the strip $S$ such that for all $T,T'
\in \overline{\T}$ holds: $T \not\simeq T'$ (Corollary \ref{cor:T}).
Then again a countability argument allows us to find  
sheared copies $\big( \begin{smallmatrix} 1 & \delta_n\\ 0 & 1 \end{smallmatrix}
\big) \overline{\T}$ of $\overline{\T}$ such that no pair of congruent tiles
occur within them, nor in between them (Lemma \ref{lem:kong-endlich}, Corollary
\ref{cor:Tn}). The tilings  $\big( \begin{smallmatrix} 1 & \delta_n\\ 0 & 1 \end{smallmatrix}
\big) \overline{\T}$ (of the strip $S$) can be stacked in order to obtain
the desired vertex-to-vertex tiling of the plane.
\begin{figure}[h]
\[ \includegraphics[width=.8\textwidth]{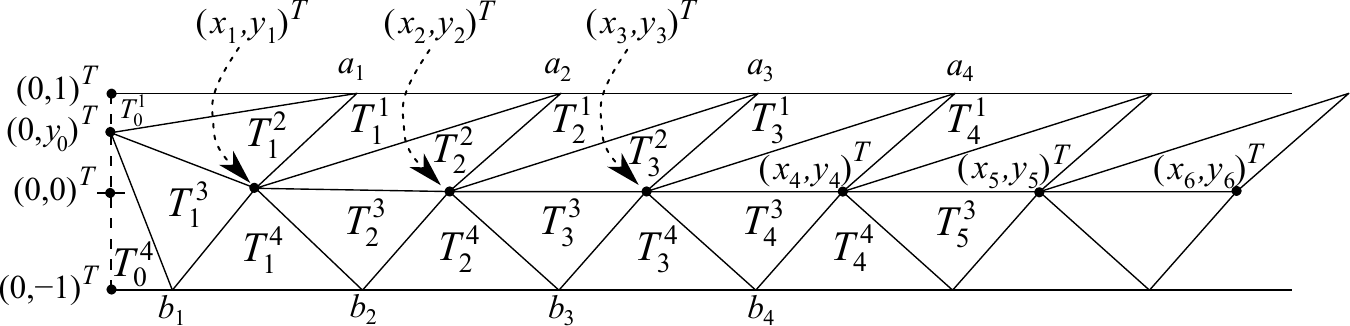} \]
\caption{
The distorted tiling of the half-strip $S^+$. The actual parameter for this
one is $y_0=\frac{3}{5}$. 
\label{fig:streifen-bez}}
\end{figure}

Let us start by considering the tilings of the strip $S$. 
Since we have mirror symmetry with respect to the vertical axis,
we first study the situation within the right half $S^+=[0,\infty) \times [-1,1]$ of $S$.
We need some notation, see Figure \ref{fig:streifen-bez}: 
Let $(x_i, y_i)^T$ denote the coordinates of 
the vertices along the central (distorted) line, separating the upper 
layer of triangles from the lower layer of triangles. Let $a_i$ denote
the $x$-coordinate of the vertices along the upper boundary of
the strip $S$ (the $y$-coordinate is always 1), and let $b_i$ denote
the $x$-coordinate of the vertices along the lower boundary of
the strip $S$ (the $y$-coordinate is always $-1$). Based on the parameter $y_0$, let
\begin{equation}\label{eq:start}
x_0=0, \quad y_0=y_0, \quad a_1 = \frac{1}{1-y_0}, \quad b_1=\frac{1}{1+y_0}
\end{equation}
and, for $i \in \N$,
\begin{align}
x_{i} &= x_{i-1}+2- \frac{2(a_{i}-b_{i})y_{i-1}}{(a_{i}-x_{i-1})(1+y_{i-1})+(b_{i}-x_{i-1})(1-y_{i-1})},\label{eq:xi}\\
y_{i} &= y_{i-1}-\frac{4y_{i-1}}{(a_{i}-x_{i-1})(1+y_{i-1})+(b_{i}-x_{i-1})(1-y_{i-1})},\label{eq:yi}\\
a_{i+1} &= a_i + \frac{2}{1-y_i},\label{eq:ai}\\
b_{i+1} &= b_i + \frac{2}{1+y_i}.\label{eq:bi}
\end{align}
The choice of $a_1$ and $b_1$ ensures that the triangles $T^1_0$ and $T^4_0$ have area $1$. Formulas \eqref{eq:xi}, \eqref{eq:yi}, \eqref{eq:ai} and \eqref{eq:bi} show that $T^1_i$, $T^2_i$, $T^3_i$ and $T^4_i$ are of unit area: a simple computation yields that they imply
\begin{align}
1 &=\frac{1}{2}\det\big((x_{i},y_{i})^T-(x_{i-1},y_{i-1})^T,(a_{i},1)^T-(x_{i-1},y_{i-1})^T\big),\label{eq:T_i'}\\
1 &=\frac{1}{2}\det\big((b_{i},-1)^T-(x_{i-1},y_{i-1})^T,(x_{i},y_{i})^T-(x_{i-1},y_{i-1})^T\big),\label{eq:D_i'}\\
1 &=\frac{1}{2}(a_{i+1}-a_i)(1-y_i),\label{eq:T_i}\\
1 &=\frac{1}{2}(b_{i+1}-b_i)(1+y_i)\label{eq:D_i}
\end{align}
for $i \in \N$. Induction shows that 
\begin{equation}\label{eq:denominator}
4i-3=\frac{1}{2}\big((x_{i-1}+a_i)(1-y_{i-1})+(x_{i-1}+b_i)(1+y_{i-1})\big)
\end{equation}
for $i \in \N$. Indeed, \eqref{eq:start} gives \eqref{eq:denominator} for $i=1$, and adding \eqref{eq:T_i'}, \eqref{eq:D_i'}, \eqref{eq:T_i} and \eqref{eq:D_i} to \eqref{eq:denominator} yields \eqref{eq:denominator} with $i$ replaced by $i+1$. By \eqref{eq:denominator}, the denominator in \eqref{eq:xi} and \eqref{eq:yi} coincides with $2(-4i+3+a_i+b_i)$. Thus
\begin{equation} \label{eq:rek-xi-yi}
x_{i} = x_{i-1}+2- \frac{(a_{i}-b_{i})y_{i-1}}{-4i+3+a_i+b_i},\quad
y_{i} = y_{i-1}-\frac{2y_{i-1}}{-4i+3+a_i+b_i}.
\end{equation}
For the sake of simplicity let $\alpha_i=a_i-(2i-1)$, $\beta_i=b_i-(2i-1)$, $\xi_i=x_i-2i$
denote the deviations of $a_i, b_i, x_i$ in the distorted tiling from 
the corresponding values in the undistorted situation. Then 
\begin{equation}\label{eq:dstart}
\xi_0=0, \quad y_0=y_0, \quad \alpha_1=\frac{y_0}{1-y_0}, \quad \beta_1=\frac{-y_0}{1+y_0}
\end{equation}
and, for $i \in \N$, 
\begin{align}
\xi_i &= \xi_{i-1} - \frac{(\alpha_i-\beta_i) y_{i-1}}{1+\alpha_i+ \beta_i},\label{eq:dxi}\\
y_i &= y_{i-1} - \frac{2y_{i-1}}{1+\alpha_i+ \beta_i}\label{eq:dyi},\\
\alpha_{i+1} &= \alpha_i + \frac{2 y_i}{1-y_i},\label{eq:dai}\\
\beta_{i+1} &= \beta_i - \frac{2 y_i}{1+y_i}.\label{eq:dbi}
\end{align}
Formulas \eqref{eq:dstart}, \eqref{eq:dai} and \eqref{eq:dbi} show that
\[
\alpha_i+\beta_i= 2 \frac{y_0^2}{1-y_0^2} + 4 \left( \frac{y_1^2}{1-y_1^2} + 
\cdots + \frac{y_{i-1}^2}{1-y_{i-1}^2} \right).
\] 
For the sake of briefness let 
\begin{equation}\label{eq:hi}
h_i=1+\alpha_{i+1}+\beta_{i+1}=1+2 \frac{y_0^2}{1-y_0^2} + 4 \left( \frac{y_1^2}{1-y_1^2} + 
\cdots + \frac{y_i^2}{1-y_i^2} \right)  
\end{equation}
for $i=0,1,\ldots$ Then \eqref{eq:dyi} becomes
\[ y_{i+1} = \left( 1- \frac{2}{h_i} \right) y_i \quad\text{ with }\quad h_{i+1}=h_i + 4 
\frac{y_{i+1}^2}{1-y_{i+1}^2}, \quad h_0 = 1+2 \frac{y_0^2}{1-y_0^2}. \]

\begin{lem} \label{lem:xi-yi-ai-bi}
If 
\begin{equation}\label{eq:cond-y0}
\frac{1}{\sqrt{3}} < y_0 < \frac{3}{\sqrt{19}}
\end{equation}
then, for all $i \ge 0$,
\begin{enumerate}
\item[(a$_i$)] $0 < y_i < 1$,
\item[(b$_i$)] $y_i < \frac{1}{2} y_{i-1}$ (no claim if $i= 0$),
\item[(c$_i$)] $2 < h_{i-1} < h_i$ (the claim means only $h_0 > 2$ if $i=0$),
\item[(d$_i$)] $h_i \le 1+\frac{10-4^{1-i}}{3} \frac{y_0^2}{1-y_0^2}$.
\end{enumerate} 
In particular, the sequence $(y_i)_{i \ge 0}$ is positive, strictly decreasing with $\lim_{i \to \infty} y_i =0$ and
\begin{equation}\label{eq:lem-yi}
0 < y_i < 2^{-1}y_{i-1} < 2^{-2} y_{i-2} < \cdots < 2^{-i} y_0 < y_0 < 1 \quad\mbox{ for }\quad i \in \N. 
\end{equation}
\end{lem}

\begin{proof}
The claims are proved by induction over $i$. Base case:

(a$_0$): $0 < y_0 < 1$ by \eqref{eq:cond-y0}.

(b$_0$): There is nothing to show. ((b$_0$) is not needed in the sequel.)

(c$_0$): $h_0=1+ 2 \frac{y_0^2}{1-y_0^2} = 1+2 \frac{1}{\frac{1}{y_0^2}-1} > 
1+2\frac{1}{\frac{1}{1/3}-1} = 2$, because $\frac{1}{\sqrt{3}} < y_0 < 1$. 

(d$_0$): $h_0=1+ 2 \frac{y_0^2}{1-y_0^2} = 1+ \frac{10-4^{1-1\cdot 0}}{3} \frac{y_0^2}{1-y_0^2}$. 

Step of induction:

(a$_{i+1}$) and (b$_{i+1}$): 
$2 \stackrel{({\rm c}_i)}{<} h_i 
\stackrel{({\rm d}_i)}{\le} 1+\frac{10-4^{1-i}}{3} \frac{y_0^2}{1-y_0^2}
< 1+\frac{10}{3} \frac{y_0^2}{1-y_0^2}
\stackrel{\eqref{eq:cond-y0}}{<} 4$. 
It follows that $0<1-\frac{2}{h_i} <\frac{1}{2}$, hence $y_{i+1} = \left(1-\frac{2}{h_i}\right) y_i 
\stackrel{({\rm a}_{i})}{\in} \left(0,\frac{1}{2} y_i\right)$. This gives (a$_{i+1}$) and (b$_{i+1}$).

(c$_{i+1}$): $2 \stackrel{({\rm c}_{i})}{<} h_i \stackrel{({\rm a}_{i+1})}{<}  h_i + 4 \frac{y_{i+1}^2}{1-y_{i+1}^2} = h_{i+1}$.

(d$_{i+1}$): We have $0 \stackrel{({\rm a}_{i+1})}{<} y_{i+1} \stackrel{({\rm b}_{i+1})}{<} 2^{-1} y_i \stackrel{({\rm b}_i)}{<} \cdots \stackrel{({\rm b}_1)}{<} 2^{-(i+1)} y _0 < y_0 < 1$. This yields
$h_{i+1} = h_i + 4 \frac{y_{i+1}^2}{1-y_{i+1}^2} 
\le h_i + 4 \frac{y_{i+1}^2}{1-y_{0}^2}
\stackrel{({\rm d}_{i})}{\le} \left(1+\frac{10-4^{1-i}}{3} \frac{y_0^2}{1-y_0^2}\right) + 4 \frac{\left(2^{-(i+1)} y_0\right)^2}{1-y_0^2}
=  1 +\frac{10-4^{1-(i+1)}}{3}\frac{y_0^2}{1-y_0^2}$.
\end{proof}

\begin{rem} \label{rem:sqrt3}
One may also choose other values $0<y_0< 1$ as the initial value.
Numerical computations suggest that in the case $0 < y_0 < \frac{1}{\sqrt{3}}$ the recursion formulas 
\eqref{eq:start}-\eqref{eq:bi} yield values $y_i$ that converge to 0 alternatingly. A corresponding tiling is shown in Figure
\ref{fig:streifen-altern}. This one may serve as well as a starting point
for constructing a vertex-to-vertex equipartition of the plane into triangles. However,
the mix of signs will make the computations more tedious.
\begin{figure}
\[ \includegraphics[width=.8\textwidth]{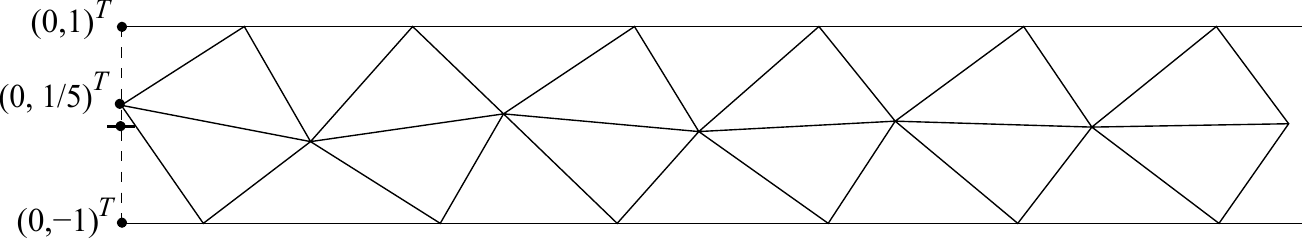} \]
\caption{Values $y_0<\frac{1}{\sqrt{3}}$ yield a tiling of the strip $S^+$
where the sign of the $y_i$ alternates. \label{fig:streifen-altern}}
\end{figure}
For $\frac{3}{\sqrt{19}} \le y_0 <1$, numerical evidence shows
that the behaviour of the resulting tilings is similar to Figure 
\ref{fig:streifen-bez}, but the proof of Lemma~\ref{lem:xi-yi-ai-bi} does not 
work in the same way as above.
The critical case $y_0=\frac{1}{\sqrt{3}}$ is considered separately in Lemma~\ref{lem:1/sqrt(3)} and Figure~\ref{fig:1/sqrt(3)} below.
\end{rem}

\begin{lem}\label{lem:perimeter}
Let $y_0 \in \left(\frac{1}{\sqrt{3}},\frac{3}{\sqrt{19}}\right)$. Then the perimeters of the triangles in the distorted tiling $\T=\T(y_0)$ of $S$ given by $y_0$ 
are bounded by some common constant.
\end{lem}

\begin{proof}
Since the triangles of the undistorted tiling from Figure~\ref{fig:streifen-ungestoert}
have constant perimeter, it is enough to show that the deviations of the coordinates of the vertices of the triangles of $\T$ from the respective coordinates from the undistorted tiling are uniformly bounded. That is, we have to show uniform boundedness of the values $\xi_i$, $y_i$, $\alpha_i$ and $\beta_i$.

Lemma~\ref{lem:xi-yi-ai-bi} (a$_i$) settles the claim for $y_i$.
For $\alpha_i$, $\beta_i$ and $\xi_i$, $i \ge 1$, we estimate
\begin{align*}
|\alpha_i| &
\stackrel{\eqref{eq:dstart},\eqref{eq:dai}}{=} \left|\frac{y_0}{1-y_0}+\sum_{j=1}^{i-1} \frac{2y_j}{1-y_j}\right|
\le 2 \sum_{j=0}^\infty \frac{|y_j|}{|1-y_j|}
\stackrel{\eqref{eq:lem-yi}}{\le} 2 \sum_{j=0}^\infty \frac{2^{-j}y_0}{1-y_0}
=\frac{4y_0}{1-y_0}=:C_\alpha,
\\
|\beta_i| & 
\stackrel{\eqref{eq:dstart},\eqref{eq:dbi}}{=} \left|-\frac{y_0}{1+y_0}-\sum_{j=1}^{i-1} \frac{2y_j}{1+y_j}\right|
\le 2 \sum_{j=0}^\infty \frac{|y_j|}{|1+y_j|}
\stackrel{\eqref{eq:lem-yi}}{\le} 2 \sum_{j=0}^\infty \frac{2^{-j}y_0}{1}
=4y_0=:C_\beta,
\\
|\xi_i| &
\stackrel{\eqref{eq:dstart},\eqref{eq:dxi}}{=} \left|-\sum_{j=1}^{i} \frac{(\alpha_j-\beta_j)y_{j-1}}{1+\alpha_j+\beta_j}\right|
\stackrel{\eqref{eq:hi}}{=} \left|-\sum_{j=1}^{i} \frac{(\alpha_j-\beta_j)y_{j-1}}{h_{j-1}}\right|
\le
\sum_{j=1}^{i} \frac{(|\alpha_j|+|\beta_j|)|y_{j-1}|}{|h_{j-1}|}\\
& 
\stackrel{\eqref{eq:lem-yi},({\rm c}_{j})}{\le}
\sum_{j=1}^{\infty} \frac{(C_\alpha+C_\beta)2^{-(j-1)}y_0}{2}
=(C_\alpha+C_\beta)y_0. \qedhere
\end{align*}
\end{proof}

Next we study congruence relations between triangles from tilings $\T=\T(y_0)$ of the strip $S$. Let $\cong$ denote congruence with respect to the group of Euclidean isometries (including reflections). We write $\simeq$ for congruence under the subgroup of all translations and all rotations by an angle of $180^\circ$. That is, two sets $A,B \in \mathbb{R}^2$ satisfy $A \simeq B$ if and only if there exist $s \in \{\pm 1\}$ and $t_1,t_2 \in \mathbb{R}$ such that $B=sA+(t_1,t_2)^T$.

We start with an observation on the tiling $\T^\ast=\T\big(\frac{1}{\sqrt{3}}\big)$ of $S$ based on the parameter $y_0=\frac{1}{\sqrt{3}}$. The respective triangles are denoted by $T_i^{\ast j}=T_i^j\big(\frac{1}{\sqrt{3}}\big)$, see Figure~\ref{fig:1/sqrt(3)}. Here $T^{*j}_{-i}$ denotes the image of
$T^{*j}_i$ under reflection in the vertical axis.

\begin{figure}
\[ \includegraphics[width=.8\textwidth]{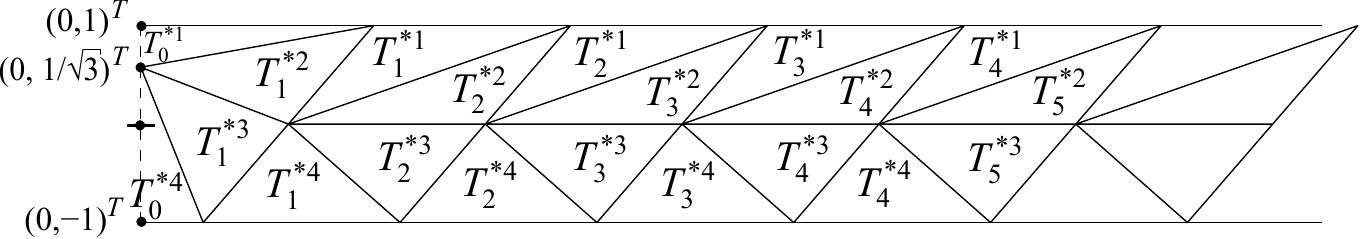} \]
\caption{The tiling $\T^\ast$ of $S$ with parameter $y_0=\frac{1}{\sqrt{3}}$.
  The six congruence classes of tiles in $\T^*$ are $[T_0^{*1}],  [T_0^{*4}],  [T_1^{*2}],
  [T_1^{*3}]$ (one element each), $[T_1^{*1}]=[T_2^{*2}],  [T_1^{*4}]=[T_2^{*3}]$.
  \label{fig:1/sqrt(3)}}
\end{figure}

\begin{lem}\label{lem:1/sqrt(3)}
The coordinates of the tiling $\T^\ast$ are $x_0=0$, $y_0=\frac{1}{\sqrt{3}}$ and $x_i=2i-\frac{1}{2}$, $y_i=0$, $a_i=2i+\frac{\sqrt{3}-1}{2}$, $b_i=2i-\frac{\sqrt{3}+1}{2}$ for $i \ge 1$. Every triangle of $\T^\ast$ is congruent to one of $T_0^{\ast 1}$, $T_1^{\ast 1}$,  $T_1^{\ast 2}$, $T_1^{\ast 3}$, $T_0^{\ast 4}$ and $T_1^{\ast 4}$.
Moreover,
\begin{equation}\label{eq:simeq_ast}
T_i^{\ast j} \not \simeq T_{-i}^{\ast j} \quad\mbox{ for }\quad i \in \N,\, j=1,2,3,4.
\end{equation}
\end{lem}

\begin{proof}
This is a direct consequence of \eqref{eq:start}-\eqref{eq:bi}.
\end{proof}

\begin{lem}\label{lem:F(i,j,i',j')}
For $y_0 \in (0,1)$ we denote the triangles from the tiling $\T=\T(y_0)$ of $S$ by $T_i^j=T_i^j(y_0)$, $(i,j) \in I=\big((\mathbb{Z} \setminus \{0\}) \times \{1,2,3,4\}\big) \cup \{(0,1),(0,4)\}$. If $(i,j),(i',j') \in I$ are such that $T_i^{\ast j} \not\simeq T_{i'}^{\ast j'}$, then the set
\[
F(i,j,i',j')=\left\{y_0 \in (0,1)\middle|\, T_i^j(y_0) \simeq T_{i'}^{j'}(y_0)\right\}
\]
is finite.
\end{lem}

\begin{proof}
We describe the triangles by their vertices, that is, $T_i^j=\triangle\left((v^1_1,v^1_2)^T,(v^2_1,v^2_2)^T,(v^3_1,v^3_2)^T\right)$ and $T_{i'}^{j'}=\triangle\left((v'^1_1,v'^1_2)^T,(v'^2_1,v'^2_2)^T,(v'^3_1,v'^3_2)^T\right)$.
Assume that $T_i^j \simeq T_{i'}^{j'}$. Then there are $s \in \{\pm 1\}$ and $t_1,t_2 \in \mathbb{R}$ such that $T_{i'}^{j'}=sT_i^j+(t_1,t_2)^T$. The corresponding map $\varphi\big((x,y)^T\big)=s(x,y)^T+(t_1,t_2)^T$ induces a permutation $\pi$ of $\{1,2,3\}$ via $\varphi\left( (v^k_1,v^k_2)^T\right)=\big(v'^{\pi(k)}_1,v'^{\pi(k)}_2\big)^T$. Thus
\begin{equation}\label{eq:map}
s\left(v^k_1,v^k_2\right)^T+(t_1,t_2)^T=\big(v'^{\pi(k)}_1,v'^{\pi(k)}_2\big)^T \quad\mbox{ for }\quad k=1,2,3.
\end{equation}
Now we distinguish 12 situations depending on the choice of $s \in \{\pm 1\}$ and the permutation $\pi$.

\emph{Case 1: $s=1$ and $\pi$ is the identity. } From \eqref{eq:map} with $k=1$ we obtain $t_1=v'^1_1-v_1^1$ and $t_2=v'^1_2-v^1_2$. Substituting these into \eqref{eq:map} gives
\[ 
\left(v^k_1,v^k_2\right)^T+\left(v'^1_1-v_1^1,v'^1_2-v^1_2\right)^T=\left(v'^{k}_1,v'^{k}_2\right)^T \quad\mbox{ for }\quad k=1,2,3.
\]
These are six linear equations in terms of coordinates of vertices of $T_i^j$ and $T_{i'}^{j'}$. By \eqref{eq:start}-\eqref{eq:bi} these coordinates are rational functions of $y_0$. Since $T_i^{\ast j} \not\simeq T_{i'}^{\ast j'}$, at least one of these six equations fails when $y_0$ is replaced by $\frac{1}{\sqrt{3}}$. Hence that very equation is a non-trivial rational equation in $y_0$, that may have at most finitely many solutions $y_0$. Thus Case 1 applies to at most finitely many elements of $F(i,j,i',j')$.

\emph{Case 2: $s=-1$ and $\pi$ is the identity. } Now \eqref{eq:map} with $k=1$ gives
$t_1=v'^1_1+v_1^1$ and $t_2=v'^1_2+v^1_2$. We obtain
\[
-\left(v^k_1,v^k_2\right)^T+\left(v'^1_1+v_1^1,v'^1_2+v^1_2\right)^T=\left(v'^{k}_1,v'^{k}_2\right)^T \quad\mbox{ for }\quad k=1,2,3
\]
and follow the same arguments as above.
In the same way we see that each of the 12 cases yields only finitely many elements of  $F(i,j,i',j')$.
\end{proof}

\begin{lem}\label{lem:Ffinite}
The set
\[
F=\left\{y_0 \in \left(\frac{1}{\sqrt{3}},\frac{3}{\sqrt{19}}\right)\middle| \mbox{ There are distinct triangles } T,T' \in \T(y_0) \mbox{ such that } T \simeq T'.\right\}
\]
is at most countable. 
\end{lem}

\begin{proof}
Using notation from Lemma~\ref{lem:F(i,j,i',j')} we have
\[
F= \bigcup_{(i,j),(i',j') \in I, (i,j) \ne (i',j')} F(i,j,i',j') \cap \left(\frac{1}{\sqrt{3}},\frac{3}{\sqrt{19}}\right).
\]
We shall see that all the sets $F(i,j,i',j') \cap \left(\frac{1}{\sqrt{3}},\frac{3}{\sqrt{19}}\right)$ are finite. 

Let $y_0 \in \left(\frac{1}{\sqrt{3}},\frac{3}{\sqrt{19}}\right)$. The \emph{vertical width} $\operatorname{vw}(T)$ of a triangle $T$ is the maximal distance between second coordinates of vertices of $T$. By Lemma~\ref{lem:xi-yi-ai-bi}, $(y_i)_{i \ge 0}$ is positive and strictly decreasing. We obtain
\[
\operatorname{vw}\left(T_i^j\right)
=\left\{ 
\begin{array}{ll}
1-y_{|i|}, & j \in \{1,2\},\\
1+y_{|i|-1}, & j=3,\\
1+y_{|i|}, & j=4.
\end{array}
\right.
\]
See Figure~\ref{fig:streifen-bez} for an illustration. 
Note that $T_i^j \simeq T_{i'}^{j'}$ implies $\operatorname{vw}\left(T_i^j\right)=\operatorname{vw}\left(T_{i'}^{j'}\right)$.

Now assume that $y_0 \in F(i,j,i',j')$, that is, $T_i^j \simeq T_{i'}^{j'}$.

\emph{Case 1: $j \in \{1,2\}$ and $j' \in \{3,4\}$ (resp.\ $j' \in \{1,2\}$ and $j \in \{3,4\}$). }
We see that $\operatorname{vw}\left(T_i^j\right) \ne \operatorname{vw}\left(T_{i'}^{j'}\right)$, which contradicts $T_i^j \simeq T_{i'}^{j'}$.
Hence $F(i,j,i',j') \cap \left(\frac{1}{\sqrt{3}},\frac{3}{\sqrt{19}}\right)$ is empty.

\emph{Case 2: $\{j,j'\}=\{1,2\}$ or $\{j,j'\}=\{3,4\}$. } Then one of $T_i^j$ and $T_{i'}^{j'}$ has a horizontal edge and the other one has not, again a contradiction to $T_i^j \simeq T_{i'}^{j'}$. So $F(i,j,i',j') \cap \left(\frac{1}{\sqrt{3}},\frac{3}{\sqrt{19}}\right)$ is empty, too.

\emph{Case 3: $j=j'$ and $|i| \ne |i'|$. } We can argue as in Case 1.

\emph{Case 4: $j=j'$ and $|i|=|i'|$. } Since  $(i,j) \ne (i',j')$, we have $i'=-i \ne 0$. By \eqref{eq:simeq_ast} and Lemma~\ref{lem:F(i,j,i',j')}, $F(i,j,i',j')$ is finite. This completes the proof.
\end{proof}

Picking $y_0 \in \left(\frac{1}{\sqrt{3}},\frac{3}{\sqrt{19}}\right) \setminus F$ and 
using Lemma~\ref{lem:perimeter} yields the following result.

\begin{cor}\label{cor:T}
There is a vertex-to-vertex tiling $\overline{\T}$ of the strip $S$ by triangles of unit area and uniformly bounded perimeter such that $T \not \simeq T'$ for all $T,T' \in \overline{\T}$ with $T \ne T'$.
\end{cor}

The final tiling of $\mathbb{R}^2$ will be obtained by stacking sheared copies $\left( \begin{smallmatrix}1 & \delta \\ 0 & 1 
\end{smallmatrix} \right) \overline{\T}$ of $\overline{\T}$; compare Figure \ref{fig:final}.  In order to make sure that almost every shear mapping of $\overline{\T}$ produces 
\begin{itemize}
\item[(i)] mutually incongruent triangles that are 
\item[(ii)] different from countably many prescribed shapes, 
\end{itemize}
we need the following result.
\begin{figure}
\[ \includegraphics[width=.7\textwidth]{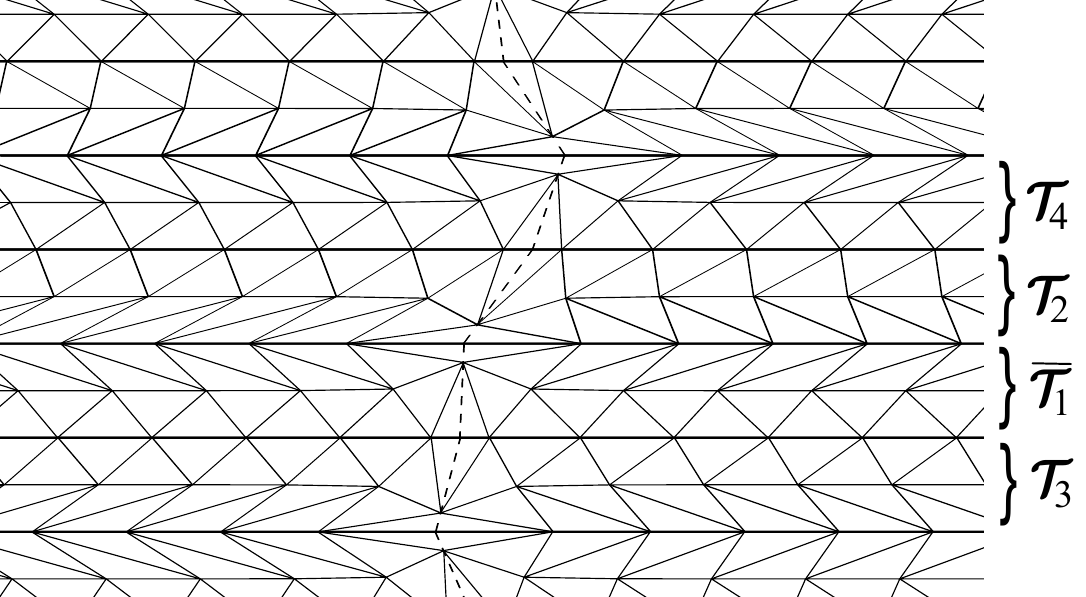} \]
\caption{Stacked images of $\overline{\T}_n$, $n \ge 1$, tile the plane. The $\overline{\T}_n$
are sheared copies of $\overline{\T}$. $\T_n$ denotes the image of $\overline{\T}_n$ under an
appropriate translation (and possibly reflection).
\label{fig:final}} 
\end{figure}
\begin{lem} \label{lem:kong-endlich}
Let $T$ and $T'$ be triangles.
\begin{enumerate}
\item[(a)] If $T \not\simeq T'$ then the set
$\left\{ \delta \in \R \mid \left( \begin{smallmatrix}1 & \delta \\ 0 & 1 
\end{smallmatrix} \right) T \cong \left( \begin{smallmatrix}1 & \delta \\ 0 & 1 
\end{smallmatrix} \right) T' \right\}$ is finite.
\item[(b)] The set $\left\{ \delta \in \R \mid
\left( \begin{smallmatrix}1 & \delta \\ 0 & 1 \end{smallmatrix} \right) T
\cong T' \right\}$ is finite.
\end{enumerate}
\end{lem}

\begin{proof}
(a) Since $T \not\simeq T'$, there exists an 
edge $e_0$ of $T$ that is a translate of neither of the three 
edges $e_1,e_2,e_3$ of $T'$. Then
\[
\left\{ \delta \in \R \mid \left( \begin{smallmatrix}1 & \delta \\ 0 & 1 
\end{smallmatrix} \right) T \cong \left( \begin{smallmatrix}1 & \delta \\ 0 & 1 
\end{smallmatrix} \right) T' \right\} 
\subseteq \big\{ \delta \in \R \big| \left\|\left( \begin{smallmatrix}1 & \delta \\ 0 & 1 
\end{smallmatrix} \right) e_0 \right\| 
\in \left\{ \left\|\left( \begin{smallmatrix}1 & \delta \\ 0 & 1 
\end{smallmatrix} \right)e_i\right\| \mid i=1,2,3 \right\}
\big\}=H_1\cup H_2 \cup H_3
\]
with $H_i=\left\{ \delta \in \R \mid  \| \big( \begin{smallmatrix}1 & \delta \\ 0 & 1 
\end{smallmatrix} \big) e_0 \|^2 = \| \big( \begin{smallmatrix}1 & \delta \\ 0 & 1 
\end{smallmatrix} \big) e_i\|^2 \right\}$. We shall show that $|H_i|<\infty$.

Let $(x_i,y_i)^T$ be the vector joining the endpoints of edge $e_i$,
where $0 \le i \le 3$. Without loss of generality, $y_i \ge 0$ and $x_i > 0$ 
if $y_i=0$. Then, for $1 \le i \le 3$,
\begin{align*} 
H_i & = \left\{ \delta \in \R \mid (x_0+\delta y_0)^2 + y_0^2 = 
(x_i + \delta y_i)^2 + y_i^2 \right\} \\
&  = \left\{ \delta \in \R \mid \delta^2(y_0^2- y_i^2) + 
\delta2(x_0y_0 - x_iy_i)+x_0^2 + y_0^2 -x_i^2 -y_i^2=0 \right\}.
\end{align*}
Either $y_0 \ne y_i$, then $y_0^2-y_i^2 \ne 0$ and the last equation has at most two 
solutions in $\delta$. Or $y_0=y_i$ and $x_0 \ne x_i$. In the latter case we may have $y_0=y_i\ne 0$, whence $x_0 y_0 - x_i y_i \ne 0$ and the equation in the 
set above has a unique solution. Otherwise $y_0=y_i=0$ and $x_0 \ne x_i$. Since $x_0,x_i > 0$, the equation gives the contradiction $x_0^2-x_i^2=0$, and $H_i$ is empty. Altogether we obtain $|H_i| \le 2$.

(b) Now let $e_0$ be an edge of $T$ such that the corresponding vector $(x_0,y_0)^T$
satisfies $y_0 \ne 0$. Denote the edges of $T'$ by $e_1,e_2,e_3$ 
as above. Then
\begin{align*}
\left\{ \delta \in \R \mid \left( \begin{smallmatrix}1 & \delta \\ 0 & 1 
\end{smallmatrix} \right) T \cong T' \right\}
&\subseteq  \left\{ \delta \in \R \mid \| \left( \begin{smallmatrix}1 & \delta \\ 0 & 1 
\end{smallmatrix} \right) e_0 \|^2 \in \left\{\|e_1\|^2, \|e_2\|^2, \|e_3\|^2 \right\} \right\} \\
&= \left\{ \delta \in \R \mid (x_0+\delta y_0)^2+y_0^2   
\in \left\{\|e_1\|^2, \|e_2\|^2, \|e_3\|^2 \right\} \right\} \\
&= \left\{ \delta \in \R \mid  \delta^2y_0^2 + \delta 2 x_0 y_0 +x_0^2+y_0^2
\in \left\{\|e_1\|^2, \|e_2\|^2, \|e_3\|^2 \right\} \right\}.
\end{align*}
Since $y_0 \ne 0$, the last term is quadratic in $\delta$ again. Thus the cardinality of the last set is at most six.
\end{proof}

Now we can provide sheared images of the tiling $\overline{\T}$ that shall be stacked in order to form a tiling of the plane.

\begin{cor}\label{cor:Tn}
There exist sheared images $\overline{\T}_n= \left( \begin{smallmatrix}1 & \delta_n \\ 0 & 1 
\end{smallmatrix} \right) \overline{\T}$ of the tiling $\overline{\T}$, $n \in \N$, such that \begin{enumerate}
\item[(a)] $|\delta_n|<1$ for all $n \ge 1$,
\item[(b)] for all $n \ge 1$, $\overline{\T}_n$ does not contain two distinct congruent triangles,
\item[(c)] for all $1 \le n' < n$, there are no congruent triangles $T \in \overline{\T}_n$ and $T' \in \overline{\T}_{n'}$.
\end{enumerate} 
\end{cor}

\begin{proof}
We construct the tilings $\overline{\T}_n$ by induction over $n$.

Base case (construction of $\overline{\T}_1$): By Corollary~\ref{cor:T} and Lemma~\ref{lem:kong-endlich} (a), the set
\[
A= \bigcup_{T,T' \in \overline{\T},T \ne T'} \left\{\delta \in \mathbb{R} \mid
\left(\begin{smallmatrix}1 & \delta \\ 0 & 1 
\end{smallmatrix} \right)T\cong\left(\begin{smallmatrix}1 & \delta \\ 0 & 1 
\end{smallmatrix} \right)T' \right\}
\]
is at most countable, because $\overline{\T}$ is countable. We pick $\delta_1 \in (-1,1) \setminus A$ and obtain $\overline{\T}_1=\left(\begin{smallmatrix}1 & \delta_1 \\ 0 & 1 
\end{smallmatrix} \right)\overline{\T}$. The choice of $\delta_1$ implies claims (a) and (b) for $n=1$.

Step of induction (construction of $\overline{\T}_n$, $n \ge 2$): By Lemma~\ref{lem:kong-endlich} (b), the set
\[
B=\bigcup_{T \in \overline{\T},T' \in \overline{\T}_1 \cup \ldots \cup \overline{\T}_{n-1}} \left\{\delta \in \mathbb{R} \mid
\left(\begin{smallmatrix}1 & \delta \\ 0 & 1 
\end{smallmatrix} \right)T \cong T' \right\}
\]
is at most countable, since $\overline{\T}$ and $\overline{\T}_1 \cup \ldots \cup \overline{\T}_{n-1}$ are countable. We fix some $\delta_n \in (-1,1) \setminus (A \cup B)$ this way defining $\overline{\T}_n=\left(\begin{smallmatrix}1 & \delta_n \\ 0 & 1 
\end{smallmatrix} \right)\overline{\T}$. We obtain the respective items of condition (a) by $\delta_n \in (-1,1)$, of (b) by $\delta_n \notin A$, and of (c) by $\delta_n \notin B$.
\end{proof}

All tilings $\overline{\T}_n=\left(\begin{smallmatrix}1 & \delta_n \\ 0 & 1 
\end{smallmatrix} \right)\overline{\T}$ of $S$ have the same mutual distances $a_{i+1}-a_i$ between adjacent vertices at the upper boundary $y=1$ as $\overline{\T}$. The same applies to the lower boundary $y=-1$. $S$ is tiled by $\overline{\T}_1$. Therefore the parallel strip $S+(0,2)^T$ can be tiled by a suitable image $\T_2$ of $\overline{\T}_2$ under a reflection 
with respect to the horizontal axis and some translation such that $\overline{\T}_1 \cup \T_2$ is vertex-to-vertex. Similarly, we tile $S+(0,-2)^T$ by a reflected and translated image $\T_3$ of $\overline{\T}_3$, we tile $S+(0,4)^T$ and $S+(0,-4)^T$ by suitable translates $\T_4$ and $\T_5$ of $\overline{\T}_4$ and $\overline{\T}_5$ etc. This way we obtain the desired vertex-to-vertex tiling $\overline{\T}_1 \cup \T_2 \cup \T_3 \cup \ldots$ of $\mathbb{R}^2$, see Figure~\ref{fig:final}. The perimeters of all triangles of that tiling are uniformly bounded because of the respective property of $\overline{\T}$ (see Corollary~\ref{cor:T}) and by Corollary~\ref{cor:Tn}~(a).    

The proof of Theorem~\ref{thm:vtv-triangles} is complete.
\end{proof}


\section{Non-vtv equipartitions of the plane into hexagons}

\begin{figure}[b]
\[ \includegraphics[width=.7\textwidth]{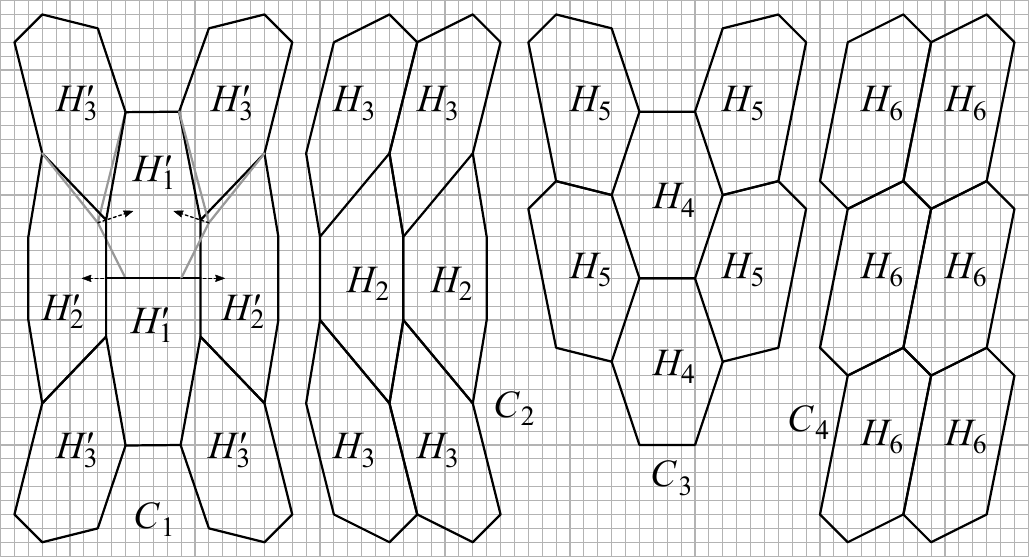} \]
\caption{Copies of these four clusters are used to tile the plane.
In each copy the hexagons will be distorted in a different way. 
\label{fig:6ecke-cluster}} 
\end{figure}
\begin{thm}
There is a non-vertex-to-vertex tiling of the plane by pairwise incongruent 
convex hexagons of unit area and uniformly bounded perimeter.
\end{thm}
\begin{proof}
The construction uses a variant of the tiling-a-tile method from
Subsection \ref{sec:tiling-a-tile}. Here we need four different clusters
$C_1, C_2, C_3, C_4$ that 
\begin{itemize}
\item[(i)] can tile the plane, and 
\item[(ii)] can be dissected into convex hexagons of unit area in uncountably many ways.
\end{itemize}
Figure \ref{fig:6ecke-cluster} shows how each of the four clusters is
dissected into copies of altogether eight different hexagons $H'_1,H'_2,H'_3,
H_2, \ldots, H_6$, each hexagon having unit area (unit area means here:
72 small boxes). 
\begin{figure}
\[ \includegraphics[width=.7\textwidth]{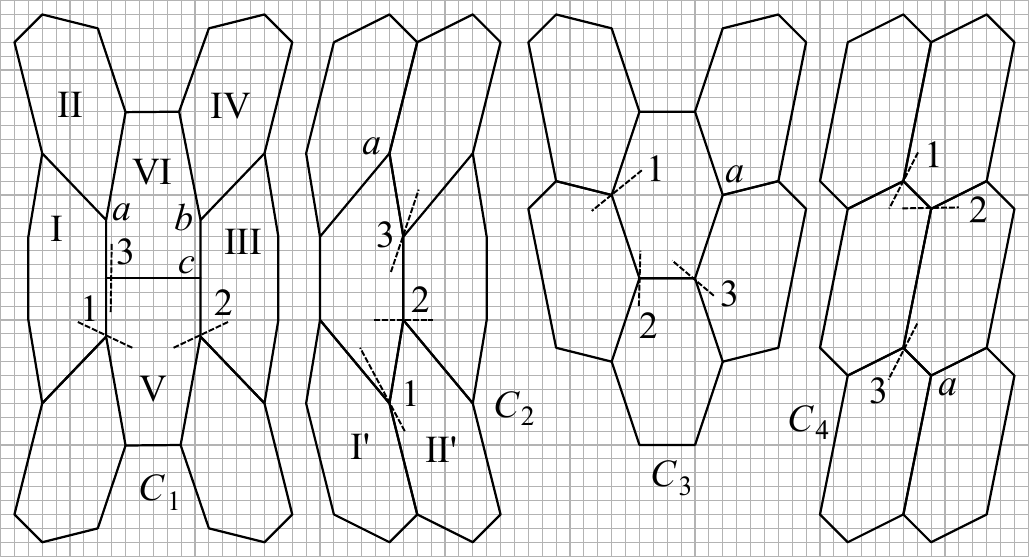} \]
\caption{Distortions within the four clusters from Figure~\ref{fig:6ecke-cluster}. Each dashed segment represents one degree of freedom. Numbers give the order of choosing free vertices.
\label{fig:6ecke-distortions}} 
\end{figure}
Figure \ref{fig:6ecke-distortions} illustrates the degrees 
of freedom for distortions within each cluster. In order to see that
the distortions act as desired one needs to consider the dependencies
in each cluster. For instance, in $C_1$, vertex 1 can be shifted
continuously along the dashed line, hence produces uncountably
many distinct versions of the lower left hexagon. Once vertex 1 is fixed,
the position of the vertex marked by $a$ is determined uniquely by
the requirements that hexagons I and II both need to have unit area.
Independently from
this choice, vertex 2 can be shifted along the dashed line, fixing
vertex $b$ by the requirement that hexagons III and IV both have unit
area, too. Independent of these choices vertex 3 can be shifted
along the approximately vertical dashed line. Vertex $c$ is then determined
uniquely by the requirement that hexagons V and VI have unit area.
Note also that indeed all hexagons will be continuously distorted
under these transformations.

In a similar manner one may shift vertex 1 in cluster $C_2$ by some
small amount along the dashed line, preserving the area of hexagon I'.
Once vertex 1 is fixed this determines the vertical position of the dashed
line at vertex 2 by the requirement that hexagon II' has unit area.
Still vertex 2 can be shifted along that
line in a continuous manner. The choice of 2 determines the exact
position of the dashed line at vertex 3. Still vertex 3 has one degree
of freedom along this line. The position of 3 determines the position of
vertex $a$ uniquely. Again note that all six hexagons will be distorted
continuously under these transformations.
In a similar manner the reader can convince oneself that the
same is true for clusters $C_3$ and $C_4$. 
\begin{figure}[b]
\[ \includegraphics[width=.9\textwidth]{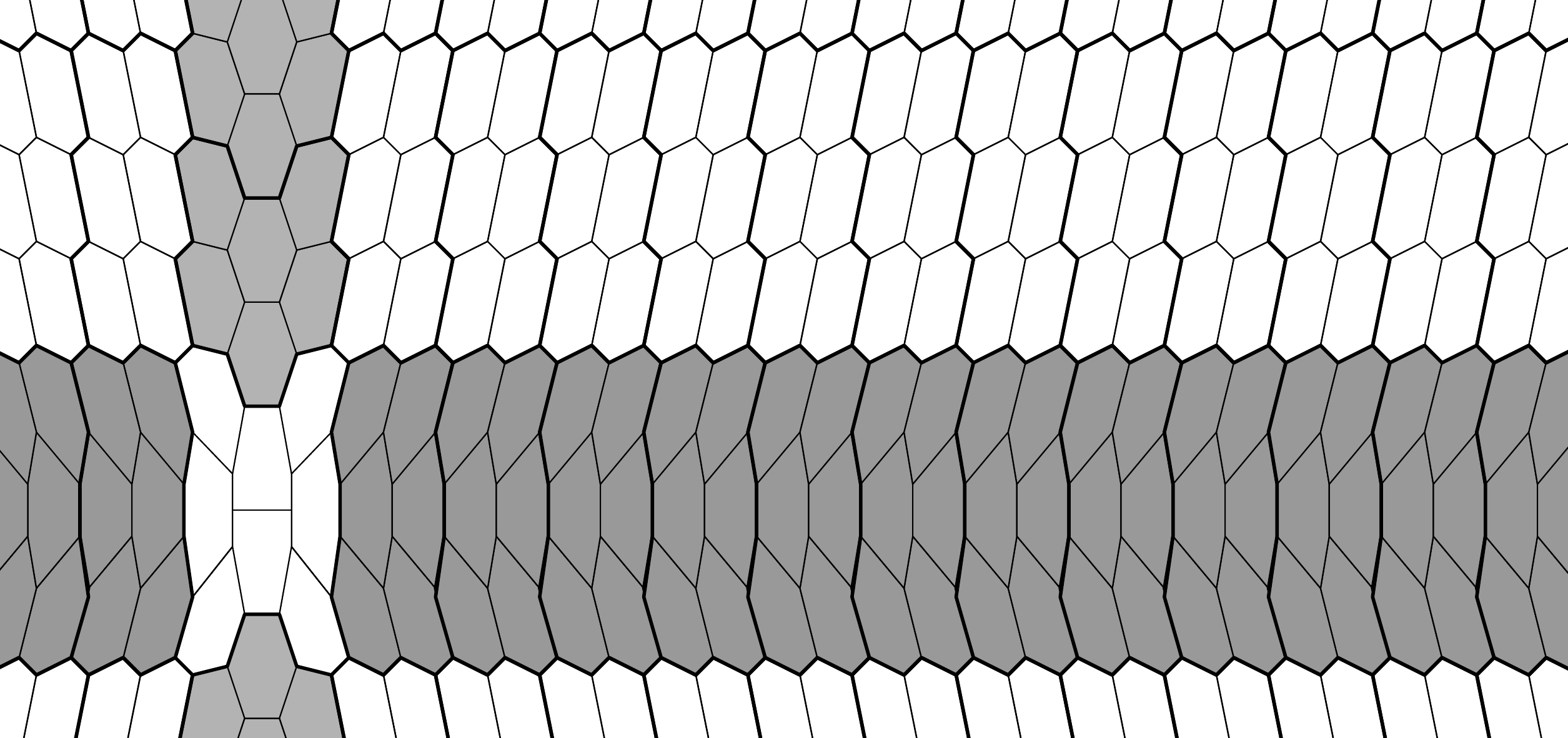} \]
\caption{The image illustrates how copies of the four clusters 
are used to tile the plane. \label{fig:6ecke-tiling}} 
\end{figure}

The desired tiling is sketched in Figure \ref{fig:6ecke-tiling}.
It can be obtained in a similar manner inductively as in the last section:
Start with the central cluster $C_1$, distort the hexagons such that
no two of them are congruent to each other. Add the next cluster,
distort the hexagons such that no pairwise congruent hexagons occur.
In each step there are only finitely many shapes to avoid, whereas
there are uncountably many hexagon shapes available.

Note that the cluster $C_1$ is used only once in the tiling.
This cluster contains the only points where the tiling is not vertex-to-vertex.
\end{proof}


\section*{Acknowledgments}
Both authors express their gratitude to R.\ Nandakumar for providing 
several interesting problems. The first author thanks Friedrich Schiller
University of Jena for financial support.


\end{document}